\documentclass[10pt]{amsart}
\usepackage{amsmath,amscd}
\usepackage{amssymb}

\usepackage{times}

\usepackage[T1]{fontenc}
\usepackage{graphicx}
\usepackage[svgnames]{xcolor}
\usepackage{framed}

\usepackage{tikz}
\usepackage{tikz-cd}

\author{Liran Shaul}
\address{Universiteit Antwerpen, Departement Wiskunde-Informatica, Middelheim campus,
Middelheimlaan 1,
2020 Antwerp, Belgium}

\curraddr{Fakult\"at f\"ur Mathematik\\ 
Universit\"at Bielefeld\\ 
33501 Bielefeld\\ 
Germany.}

\email{LShaul@math.uni-bielefeld.de}
\newtheorem{thm}[equation]{Theorem}

\newtheorem{lem}[equation]{Lemma}

\theoremstyle{definition}

\newtheorem{rem}[equation]{Remark}

\newcommand{\iso}{\xrightarrow{\simeq}}
\newcommand{\inj}{\hookrightarrow}

\newcommand{\opn}{\operatorname}
\newcommand{\cat}[1]{\operatorname{\mathsf{#1}}}

\newcommand{\mfrak}[1]{\mathfrak{#1}}

\newcommand{\mrm}[1]{\mathrm{#1}}
\newcommand{\mbb}[1]{\mathbb{#1}}

\renewcommand{\k}{\Bbbk}
\newcommand{\K}{\mbb{K} \hspace{0.05em}}

\renewcommand{\a}{\mfrak{a}}

\numberwithin{equation}{section} 
\thanks{The author acknowledges the support of the European Union for the ERC grant No 257004-HHNcdMir.}

\begin{document}

\title[Adic reduction to the diagonal and cofiniteness]{Adic reduction to the diagonal and a relation between cofiniteness and derived completion}

\begin{abstract}
We prove two results about the derived functor of $\a$-adic completion:
(1) Let $\K$ be a commutative noetherian ring, let $A$ be a flat noetherian $\K$-algebra which is $\a$-adically complete with respect to some ideal $\a\subseteq A$, such that $A/\a$ is essentially of finite type over $\K$,  and let $M,N$ be finitely generated $A$-modules. Then adic reduction to the diagonal holds: $A\otimes^{\mrm{L}}_{  A\widehat{\otimes}_{\K} A } ( M\widehat{\otimes}^{\mrm{L}}_{\K} N ) \cong M \otimes^{\mrm{L}}_A N$. A similar result is given in the case where $M,N$ are not necessarily finitely generated.
(2) Let $A$ be a commutative ring, let $\a\subseteq A$ be a weakly proregular ideal, let $M$ be an $A$-module, and assume that the $\a$-adic completion of $A$ is noetherian (if $A$ is noetherian, all these conditions are always satisfied). Then $\opn{Ext}^i_A(A/\a,M)$ is finitely generated for all $i\ge 0$  if and only if the derived $\a$-adic completion $\mrm{L}\widehat{\Lambda}_{\a}(M)$ has finitely generated cohomologies over $\widehat{A}$. The first result is a far reaching generalization of a result of Serre, who proved this in case $\K$ is a field or a discrete valuation ring and $A = \K[[x_1,\dots,x_n]]$.  
\end{abstract}
\maketitle

\section{Introduction}

All rings in this paper are commutative and unital (but not necessarily noetherian). 
For such a ring $A$, we denote by $\cat{Mod} A$ the category of $A$-modules, 
by $\cat{C}(A)$ the category of complexes of $A$-modules,
and by $\cat{D}(A)$ its derived category.
We use cohomological indexing, with differentials being of degree $+1$.

\subsection{Adic completion functors and their derived functors}\label{sec:wpr-intro}

Let $A$ be a commutative ring, 
and let $\a\subseteq A$ be a finitely generated ideal.
The $\a$-adic completion functor associated to $A$ and $\a$ is given by
\[
\Lambda_{\a}(-) := \varprojlim A/\a^n \otimes_A -.
\]
This is an additive functor $\cat{Mod} A \to \cat{Mod} A$. 
We will sometimes denote $\Lambda_{\a}(M)$ by $\widehat{M}$.
The derived functor
\[
\mrm{L}\Lambda_{\a} : \cat{D}(A) \to \cat{D}(A)
\]
exists. %, and can be calculated using K-flat resolutions (see \cite[Section 1]{AJL1} for a proof).
In general, the functor $\mrm{L}\Lambda_{\a}$ can behave badly, 
but if $A$ is noetherian it is known to have good behavior and an explicit formula in terms of a set of generators of $\a$.
It was realized in recent years that one can relax the noetherian assumption, 
and instead let $A$ be arbitrary and assume that $\a$ is weakly proregular.
The weak proregularity condition, a property of certain finitely generated ideals in a commutative ring, 
which we recall in Section \ref{sec:prel} below,
is always satisfied when $A$ is noetherian.
It turns out that weak proregularity is the precise condition which guarantees good behavior of $\mrm{L}\Lambda_{\a}$. 
The $A$-module $\Lambda_{\a}(A)$ has the structure of a commutative ring which we denote by $\widehat{A}$.
There is a ring homomorphism $A \to \widehat{A}$.  
In case it is bijective $A$ is called an \textbf{adic ring}.
If $A$ is noetherian then this map is flat, 
but in general $\widehat{A}$ can fail to be flat over $A$ (even if $\a$ is weakly proregular). 
For any $A$-module $M$, the $A$-module $\Lambda_{\a}(M)$ has a structure of an $\widehat{A}$-module, 
and this gives rise to an additive functor $\cat{Mod} A \to \cat{Mod} \widehat{A}$.
We denote this functor by
\[
\widehat{\Lambda}_{\a} : \cat{Mod} A \to \cat{Mod} \widehat{A},
\]
and its derived functor by
\[
\mrm{L}\widehat{\Lambda}_{\a} : \cat{D}(A) \to \cat{D}(\widehat{A}).
\]

\subsection{Two questions in commutative algebra concerning derived completion}\label{sec:quest}

\subsubsection{Adic reduction to the diagonal}

In his seminal book \cite{Se}, Serre defined the intersection multiplicity $\chi(M,N)$ of a pair of finitely generated modules $M,N$ such that $M\otimes_A N$ has finite length, over a a noetherian regular local ring $(A,\mfrak{m})$.
Serre conjectured that the number $\chi(M,N)$ satisfies various natural properties, 
and was able to prove his conjectures in case $\widehat{A} := \Lambda_{\mfrak{m}}(A)$ is of the form $\widehat{A} = \K[[x_1,\dots,x_n]]$
where $\K$ is either a field or a discrete valuation ring. 
To prove these conjectures in this case, Serre used an adic version of reduction to the diagonal.
Recall that if $\K$ is a base commutative ring, and $A$ is a flat $\K$-algebra,
the reduction to the diagonal technique is based on the natural isomorphism
\[
A \otimes^{\mrm{L}}_{A\otimes_{\K} A} (M\otimes^{\mrm{L}}_{\K} N) \cong M\otimes^{\mrm{L}}_A N
\]
which holds for any $M, N \in \cat{D}(A)$. 
Serre essentially showed in \cite[Section V.B.2]{Se} that an adic version of this holds, 
namely that
\begin{equation}\label{eqn:artd}
A \otimes^{\mrm{L}}_{A\widehat{\otimes}_{\K} A} (\mrm{L}\widehat{\Lambda}_I(M\otimes^{\mrm{L}}_{\K} N)) \cong M\otimes^{\mrm{L}}_A N
\end{equation}
in the case where $A = \K[[x_1,\dots,x_n]]$ with $\K$ a field or a discrete valuation ring, 
$M, N$ are finitely generated $A$-modules, 
\[
I = (1\otimes_{\K} x_1,\dots,1\otimes_{\K} x_n,x_1\otimes_{\K} 1,\dots,x_n\otimes_{\K} 1) \subseteq A\otimes_{\K} A
\]
and 
\[
A\widehat{\otimes}_{\K} A := \Lambda_I(A\otimes_{\K} A) \cong \K[[x_1,\dots,x_n,y_1,\dots,y_n]].
\]
See \cite[Section 0.7.7]{EGA} for a discussion about completed tensor products of adic algebras.

Because of the usefulness of the reduction to the diagonal technique, 
and the fact that the completion operation can greatly simplify the structure of a noetherian ring, 
it is natural to ask: is it possible to extend (\ref{eqn:artd}) to a wider class of adic noetherian algebras?
A positive answer will be given in Theorem \ref{thm:ard} below.

\textbf{Notice} that even though we assume that $A$ is noetherian, 
for adic rings the enveloping algebra $A\otimes_{\K} A$ is usually non-noetherian, 
so the theory of weakly proregular ideals mentioned above is crucial for answering such a question.
Likewise, this forces us in Section \ref{section:MGMComp} below to work in a non-noetherian setting.

\subsubsection{Cofiniteness and derived completion}

Let $A$ be a commutative ring, 
let $\a\subseteq A$ be a finitely generated weakly proregular ideal,
and assume that $\widehat{A} := \Lambda_{\a}(A)$ is noetherian 
(if $A$ is noetherian, all these conditions are always satisfied).
The fact that $\widehat{A}$ is noetherian is equivalent to the ring $A/\a$ being noetherian.
The $\a$-torsion functor associated to $A$ and $\a$ is the functor
\[
\Gamma_{\a}(-) := \varinjlim \opn{Hom}_A(A/\a^n,-).
\]
It is a left exact additive functor $\cat{Mod} A \to \cat{Mod} A$, 
and its derived functor is the functor $\mrm{R}\Gamma_{\a}:\cat{D}(A)\to \cat{D}(A)$.
\par
An $A$-module $M$ is called \textbf{$\a$-cofinite} if $\Gamma_{\a} (M) = M$, 
and the $A/\a$-modules
\[
\opn{Ext}^i_A(A/\a,M)
\]
are finitely generated for all $i\ge 0$.
A great deal of study was done in an attempt to understand when $A$-modules (and especially cohomology with support modules) are $\a$-cofinite.
See for instance \cite{Me} and its references. Since cohomology with support modules are automatically torsion, 
the understanding of the $\a$-cofiniteness condition focuses on the finiteness of these $\opn{Ext}$-modules.
A calculation (that is repeated in the proof of Theorem \ref{thm:cofinite} below) shows that if $\mrm{L}\widehat{\Lambda}_{\a}(M)$ has finitely generated cohomologies over $\widehat{A}$ then $\opn{Ext}^i_A(A/\a,M)$ is finitely generated for all $i\ge 0$.
It is thus natural to ask: does the converse hold? We will show in Theorem \ref{thm:cofinite} that these conditions are equivalent.

\subsection{The MGM equivalence}

We continue to assume that $A$ is a commutative ring, 
and $\a\subseteq A$ is a finitely generated weakly proregular ideal, 
and remind the reader again that in a noetherian ring every ideal is weakly proregular.
The two questions from Section \ref{sec:quest} involve the derived completion functor.
It is shown in \cite{AJL1} that for any $M \in \cat{D}(A)$, 
there is a natural map 
\begin{equation}\label{eqn:complete-map}
M \to \mrm{L}\Lambda_{\a}(M).
\end{equation} 
A complex $M$ is called \textbf{cohomologically $\a$-adically complete} if (\ref{eqn:complete-map}) is an isomorphism in $\cat{D}(A)$.
The collection of all cohomologically $\a$-adically complete complexes is full triangulated subcategory of $\cat{D}(A)$, 
denoted by $\cat{D}(A)_{\a-\opn{com}}$. 
Thus, we see that questions about (cohomologically) complete complexes are really questions about objects in $\cat{D}(A)_{\a-\opn{com}}$. 
Unfortunately, we do not know how to answer the two questions above while working in $\cat{D}(A)_{\a-\opn{com}}$.
Luckily, the Matlis-Greenlees-May (MGM) equivalence provides us with an alternative.
\par
Similarly to (\ref{eqn:complete-map}), 
for any $M \in \cat{D}(A)$ there is a natural map
\begin{equation}\label{eqn:tor-map}
\mrm{R}\Gamma_{\a}(M) \to M,
\end{equation}
and $M$ is called \textbf{cohomologically $\a$-torsion} if (\ref{eqn:tor-map}) is an isomorphism in $\cat{D}(A)$.
The collection of all cohomologically $\a$-torsion complexes is a full triangulated subcategory of $\cat{D}(A)$, 
and we denote it by $\cat{D}(A)_{\a-\opn{tor}}$. 
The MGM equivalence states that the categories $\cat{D}(A)_{\a-\opn{tor}}$ and $\cat{D}(A)_{\a-\opn{com}}$ are equivalent (see \cite[Theorem 7.11]{PSY1} for a proof, and \cite[Remark 7.14]{PSY1} for a historical survey). 
Using it, we may transfer questions in $\cat{D}(A)_{\a-\opn{com}}$ to questions in $\cat{D}(A)_{\a-\opn{tor}}$.
This is almost sufficient for the purpose of solving the above questions. 

However, as both of the questions from Section \ref{sec:quest} involve passage from the ring $A$ to its completion $\widehat{A}$, 
it turns out that the MGM equivalence is not quite sufficient for the purpose of answering them.
In section \ref{section:MGMComp} below we investigate relations between the MGM equivalence 
and the functors $\mrm{R}\widehat{\Gamma}_{\a}, \mrm{L}\widehat{\Lambda}_{\a}$. These are summarized in Remark \ref{rem:MGM} below.
Using these new results, in Section \ref{sec:ARD} we answer the first question above, establishing a general adic reduction to the diagonal natural isomorphism. Finally, in Section \ref{sec:COF} we answer the second question from above, showing that $\opn{Ext}^i_A(A/\a,M)$ is finitely generated if and only if the cohomologies of the derived completion of $M$ are finitely generated over $\widehat{A}$.

\section{Preliminaries}\label{sec:prel}

\subsection{Resolutions of unbounded complexes}

We begin by recalling some basic facts about resolutions of unbounded complexes. 
A reference for this is \cite{Sp}. Let $A$ be a commutative ring.
A complex $M \in \cat{C}(A)$ is called \textbf{K-projective} (respectively \textbf{K-injective})
if for any acyclic complex $X \in \cat{C}(A)$, the complex $\opn{Hom}_A(M,X)$ (resp. $\opn{Hom}_A(X,M)$) is acyclic.
A complex $M \in \cat{C}(A)$ is called \textbf{K-flat} if for any acyclic complex $X \in \cat{C}(A)$, 
the complex $M\otimes_A X$ is acyclic. 
By \cite[Theorem C]{Sp}, every complex has a K-projective resolution and a K-injective resolution.
A K-projective complex is K-flat, so in particular, every complex has a K-flat resolution.
By \cite[Section 1]{AJL1}, the functor $\mrm{L}\Lambda_{\a}$ can be calculated using K-flat resolutions.

\subsection{Weak proregularity}

Let $A$ be a commutative ring, and let $a \in A$.
Following \cite[Section 4]{PSY1}, 
the \textbf{infinite dual Koszul complex} associated to $A$ and $a$ is the complex
\[
0 \to A \xrightarrow{d} A[a^{-1}] \to 0
\]
concentrated in degrees $0,1$, where $d$ is the localization map.
We denote this complex by $\opn{K}^{\vee}_{\infty}(A; (a))$.
If $\mathbf{a} = (a_1,\dots,a_n)$ is a finite sequence of elements of $A$, 
we define the infinite dual Koszul complex associated to $A$ and $\mathbf{a}$ to be the complex
\[
\opn{K}^{\vee}_{\infty}(A; \mathbf{a}) := \opn{K}^{\vee}_{\infty}(A; (a_1))\otimes_A \opn{K}^{\vee}_{\infty}(A; (a_2)) \otimes_A \dots \otimes_A \opn{K}^{\vee}_{\infty}(A; (a_n)).
\]
It is a bounded complex of flat $A$-modules, so it is K-flat.
The \textbf{telescope complex} is an explicit free resolution of the infinite dual Koszul complex which we now describe.
Given $a \in A$, we let $\opn{Tel}(A;(a))$ be the complex
\[
0 \to \bigoplus_{n=0}^{\infty} A \xrightarrow{d} \bigoplus_{n=0}^{\infty} A \to 0
\]
concentrated in degrees $0,1$. 
The differential $d$ acts as follows: let $\{\delta_i|i\ge 0\}$ be the standard basis of the countably generated free $A$-module 
$\bigoplus_{n=0}^{\infty} A$. Then $d(\delta_0) = \delta_0$, and $d(\delta_i) = \delta_{i-1}-a\cdot \delta_i$ for $i\ge 1$.
Again, for a finite sequence of elements $\mathbf{a} = (a_1,\dots,a_n)$ of $A$, we let
\[
\opn{Tel}(A;\mathbf{a}) := \opn{Tel}(A;(a_1)) \otimes_A \opn{Tel}(A;(a_2)) \otimes_A \dots \otimes_A \opn{Tel}(A;(a_n)).
\]
This is a bounded complex of infinitely generated free $A$-modules, so it is K-projective.
According to \cite[Lemma 5.7]{PSY1}, there is a quasi-isomorphism
$\opn{Tel}(A;\mathbf{a}) \to \opn{K}^{\vee}_{\infty}(A; \mathbf{a})$.

Let $A,B$ be a commutative rings, $\mathbf{a}$ a finite sequence of elements of $A$,
let $f:A \to B$ be a ring homomorphism, and let $\mathbf{b} = f(\mathbf{a})$.
Then as explained in \cite{PSY1}, the infinite dual Koszul complex and the telescope complex
satisfy the base change property: there are isomorphisms of complexes of $B$-modules
\[
\opn{K}^{\vee}_{\infty}(A; \mathbf{a}) \otimes_A B \cong \opn{K}^{\vee}_{\infty}(B; \mathbf{b}), \quad
\opn{Tel}(A;\mathbf{a}) \otimes_A B \cong \opn{Tel}(B;\mathbf{b}).
\]

Let $A$ be a commutative ring,
let $\a$ be a finitely generated ideal,
and let $\mathbf{a}$ be a finite sequence of elements of $A$ that generates $\a$.
The sequence $\mathbf{a}$ is called \textbf{weakly proregular} if there is an isomorphism
\[
\mrm{R}\Gamma_{\a}(-) \cong \opn{K}^{\vee}_{\infty}(A; \mathbf{a})\otimes_A -
\]
of functors $\cat{D}(A) \to \cat{D}(A)$.
Moreover, in this case, by \cite[Corollary 5.25]{PSY1},
there is also an isomorphism
\[
\mrm{L}\Lambda_{\a}(-) \cong \opn{Hom}_A(\opn{Tel}(A;\mathbf{a}),-)
\]
of functors $\cat{D}(A) \to \cat{D}(A)$.

It turns out that this property is independent of the chosen generating set of the ideal $\a$.
Hence, the ideal $\a$ will be called weakly proregular if some (equivalently, any)
finite sequence of elements of $A$ that generates it is weakly proregular.
By \cite[Theorem 4.34]{PSY1}, in a noetherian ring every ideal is weakly proregular.
See \cite[Discussion after Theorem 2.3]{HM} for an example of a finitely generated ideal that is not weakly proregular.
See \cite{AJL1,PSY1,Sc} for more information about weak proregularity.

\section{Relations between derived torsion and derived completion}\label{section:MGMComp}

Given a commutative ring $A$ and a finitely generated ideal $\a\subseteq A$, 
the $A$-module $\Gamma_{\a}(M)$ has naturally a structure of an $\widehat{A}:=\Lambda_{\a}(A)$-module.
We obtain a functor $\cat{Mod} A \to \cat{Mod} \widehat{A}$ which we denote by $\widehat{\Gamma}_{\a}$, 
and a derived functor
\[
\mrm{R}\widehat{\Gamma}_{\a}:\cat{D}(A) \to \cat{D}(\widehat{A}).
\]

The next lemma is essentially given in \cite[Corollary 0.3.1]{AJL1} in case the given ideal is proregular.
We shall need it in the more general weakly proregular case, so we give this generalization.

\begin{lem}\label{lem:tensor-completion-with-koszul}
Let $A$ be a commutative ring, 
let $\a\subseteq A$ be a finitely generated weakly proregular ideal, 
let $\mathbf{a}$ be a finite sequence of elements of $A$ that generates $\a$, 
and let $P$ be a K-flat complex of $A$-modules.
Then the map
\[
\opn{K}^{\vee}_{\infty}(A; \mathbf{a}) \otimes_A P \to \opn{K}^{\vee}_{\infty}(A; \mathbf{a}) \otimes_A \widehat{P}
\]
obtained by tensoring the canonical map $P \to \widehat{P}$ with $\opn{K}^{\vee}_{\infty}(A; \mathbf{a})$ is a quasi-isomorphism.
\end{lem}
\begin{proof}
Since there is a quasi-isomorphism 
\[
\opn{Tel}(A;\mathbf{a}) \to \opn{K}^{\vee}_{\infty}(A; \mathbf{a}),
\]
and since $\opn{Tel}(A;\mathbf{a})$ and $\opn{K}^{\vee}_{\infty}(A; \mathbf{a})$ are both K-flat over $A$,
it is enough to show that the map
\[
\opn{Tel}(A;\mathbf{a}) \otimes_A P \to \opn{Tel}(A;\mathbf{a}) \otimes_A \widehat{P}
\]
is a quasi-isomorphism.
By \cite[Corollary 5.25]{PSY1}, 
and since $\widehat{P} = \mrm{L}\Lambda_{\a}(P)$,
this map fits into a commutative diagram
\[
\begin{tikzcd}[column sep=huge, row sep=huge]
\opn{Tel}(A;\mathbf{a}) \otimes_A P 
\arrow{d}\arrow{rd} & \\
\opn{Tel}(A;\mathbf{a}) \otimes_A \opn{Hom}_A(\opn{Tel}(A;\mathbf{a}), P)
\arrow{r} &
\opn{Tel}(A;\mathbf{a}) \otimes_A \widehat{P}
\end{tikzcd}
\]
in which the horizontal map is a quasi-isomorphism.
Moreover, by \cite[Lemma 7.6]{PSY1},
the vertical map is also a quasi-isomorphism, 
so we obtain the required result.
\end{proof}

Given a commutative ring $A$, a finitely generated ideal $\a\subseteq A$, and a complex of $A$-modules $M$, note that there is a natural 
$\widehat{A}$-linear map
\begin{equation}\label{eqn:psi-comp}
\psi:M\otimes_A \widehat{A} \to \widehat{\Lambda}_{\a}(M)
\end{equation}
given by
\[
m\otimes_A (\sum_{n=0}^{\infty} a_n) \mapsto \sum_{n=0}^{\infty} a_n\cdot m,
\]
where $a_n \in \a^n\cdot \widehat{A}$.
Applying the forgetful functor $\opn{Rest}_{\widehat{A}/A}:\cat{C}(\widehat{A}) \to \cat{C}(A)$ to $\psi$,
the completion map $M \to \Lambda_{\a}(M)$ factors as
\begin{equation}\label{eqn:composition-is-comp}
M \to M\otimes_A \widehat{A} \xrightarrow{\opn{Rest}_{\widehat{A}/A}(\psi)} \Lambda_{\a} (M),
\end{equation}
where the map $M \to M\otimes_A \widehat{A}$ is induced by the completion map $A \to \widehat{A}$.

\begin{thm}\label{thm:RGammaLLambda}
Let $A$ be a commutative ring, let $\a\subseteq A$ be a finitely generated weakly proregular ideal, set $\widehat{A}:= \Lambda_{\a}(A)$, and assume that $\widehat{\a} := \widehat{A}\cdot \a$ is also weakly proregular. Then for any $M \in \cat{D}(A)$, there is a natural isomorphism
\[
\mrm{R}\Gamma_{\widehat{\a}} ( \mrm{L}\widehat{\Lambda}_{\a} (M) ) \cong \mrm{R}\widehat{\Gamma}_{\a}(M)
\]
in $\cat{D}(\widehat{A})$.
\end{thm}
\begin{proof}
Let $P \to M$ be a K-flat resolution over $A$, let $\mathbf{a}$ be a finite sequence that generates $\a$, and let $\widehat{\mathbf{a}}$ be its image in $\widehat{A}$. 
On the one hand, since $\a$ is weakly proregular, 
by \cite[Theorem 3.2]{Sh2}, there are natural isomorphisms
\[
\mrm{R}\widehat{\Gamma}_{\a}(M) \cong \mrm{R}\widehat{\Gamma}_{\a}(P) \cong (\opn{K}^{\vee}_{\infty}(A; \mathbf{a}) \otimes_A P) \otimes^{\mrm{L}}_A \widehat{A}.
\]
Since $\opn{K}^{\vee}_{\infty}(A; \mathbf{a})$ is K-flat over $A$, 
the complex $\opn{K}^{\vee}_{\infty}(A; \mathbf{a}) \otimes_A P$ is also K-flat. 
Hence, we may replace derived tensor product by ordinary tensor product, 
so there are natural isomorphisms
\[
\mrm{R}\widehat{\Gamma}_{\a}(M) \cong \mrm{R}\widehat{\Gamma}_{\a}(P) \cong \opn{K}^{\vee}_{\infty}(A; \mathbf{a}) \otimes_A P \otimes_A \widehat{A}.
\]

On the other hand, by \cite[Corollary 4.26]{PSY1}, the base change property of the infinite dual Koszul complex, 
and the fact that $\widehat{\a}$ is weakly proregular, 
there are natural isomorphisms
\[
\mrm{R}\Gamma_{\widehat{\a}} ( \mrm{L}\widehat{\Lambda}_{\a} (M) ) \cong 
\mrm{R}\Gamma_{\widehat{\a}} ( \widehat{\Lambda}_{\a} (P) ) \cong \opn{K}^{\vee}_{\infty}(\widehat{A}; \widehat{\mathbf{a}}) \otimes_{\widehat{A}} \widehat{\Lambda}_{\a} (P) \cong \opn{K}^{\vee}_{\infty}(A; \mathbf{a}) \otimes_A \widehat{\Lambda}_{\a} (P).
\]
Hence, it is enough to show that there is a natural $\widehat{A}$-linear quasi-isomorphism
\[
 \opn{K}^{\vee}_{\infty}(A; \mathbf{a}) \otimes_A P \otimes_A \widehat{A} \to \opn{K}^{\vee}_{\infty}(A; \mathbf{a}) \otimes_A \widehat{\Lambda}_{\a} (P).
\]
By (\ref{eqn:psi-comp}), there is an $\widehat{A}$-linear map $\psi:P\otimes_A \widehat{A} \to \widehat{\Lambda}_{\a} (P)$,
so we obtain an induced map
\begin{equation}\label{eqn:for-rl}
\phi = (1_{\opn{K}^{\vee}_{\infty}(A; \mathbf{a})}\otimes_A \psi): \opn{K}^{\vee}_{\infty}(A; \mathbf{a}) \otimes_A P \otimes_A \widehat{A} \to \opn{K}^{\vee}_{\infty}(A; \mathbf{a}) \otimes_A \widehat{\Lambda}_{\a} (P).
\end{equation}
We will show that (\ref{eqn:for-rl}) is a quasi-isomorphism. Let $\opn{Rest}_{\widehat{A}/A}$ be the forgetful functor
\[
\opn{Rest}_{\widehat{A}/A}: \cat{D}(\widehat{A}) \to \cat{D}(A).
\]
It is enough to show that
\[
\opn{Rest}_{\widehat{A}/A}(\phi): \opn{K}^{\vee}_{\infty}(A; \mathbf{a}) \otimes_A P \otimes_A \widehat{A} \to \opn{K}^{\vee}_{\infty}(A; \mathbf{a}) \otimes_A \Lambda_{\a} (P)
\]
is a quasi-isomorphism. To see this, note that by (\ref{eqn:composition-is-comp}) there is a commutative diagram
\[
\begin{tikzcd}[column sep=huge, row sep=huge]
\opn{K}^{\vee}_{\infty}(A; \mathbf{a}) \otimes_A P 
\arrow{r} \arrow{d}
& \opn{K}^{\vee}_{\infty}(A; \mathbf{a}) \otimes_A P \otimes_A \widehat{A}
\arrow{ld}{\opn{Rest}_{\widehat{A}/A}(\phi)}  \\
\opn{K}^{\vee}_{\infty}(A; \mathbf{a}) \otimes_A \Lambda_{\a} (P)
\end{tikzcd}
\]
The horizontal map in this diagram is a quasi-isomorphism by applying Lemma \ref{lem:tensor-completion-with-koszul} to $M = A$, 
while the vertical map in this diagram is a quasi-isomorphisms by applying Lemma \ref{lem:tensor-completion-with-koszul} to $M = P$.
Hence, $\opn{Rest}_{\widehat{A}/A}(\phi)$ is also a quasi-isomorphism, and this implies that $\phi$ is a quasi-isomorphism.
\end{proof}

Dually to (\ref{eqn:psi-comp}) and (\ref{eqn:composition-is-comp}), 
given a commutative ring $A$, a finitely generated ideal $\a\subseteq A$, and a complex of $A$-modules $M$, 
there is a natural $\widehat{A}$-linear map
\begin{equation}\label{eqn:tor-chai}
\chi: \widehat{\Gamma}_{\a} (M) \to \opn{Hom}_A(\widehat{A},M)
\end{equation}
given by
\[
m \mapsto (f_m(a):=a\cdot m).
\]
Applying the forgetful functor $\opn{Rest}_{\widehat{A}/A}:\cat{C}(\widehat{A}) \to \cat{C}(A)$ to $\chi$,
the inclusion map 
\[
\Gamma_{\a}(M) \inj M
\]
factors as
\begin{equation}\label{eqn:composition-is-tor}
\Gamma_{\a} (M) \xrightarrow{\opn{Rest}_{\widehat{A}/A}(\chi)} \opn{Hom}_A(\widehat{A},M) \to M,
\end{equation}
where the map $\opn{Hom}_A(\widehat{A},M) \to M$ is induced by the completion map $A\to \widehat{A}$.

\begin{thm}\label{thm:LLambdaRGamma}
Let $A$ be a commutative ring, let $\a\subseteq A$ be a finitely generated weakly proregular ideal, set $\widehat{A}:= \Lambda_{\a}(A)$, and assume that $\widehat{\a} := \widehat{A}\cdot \a$ is also weakly proregular. Then for any $M \in \cat{D}(A)$, there is a natural isomorphism
\[
\mrm{L}\Lambda_{\widehat{\a}} ( \mrm{R}\widehat{\Gamma}_{\a} (M) ) \cong \mrm{L}\widehat{\Lambda}_{\a}(M)
\]
in $\cat{D}(\widehat{A})$.
\end{thm}
\begin{proof}
Let $M \to I$ be a K-injective resolution over $A$. By \cite[Theorem 3.6]{Sh2}, 
since $\a$ is weakly proregular,
there are natural isomorphisms
\[
\mrm{L}\widehat{\Lambda}_{\a}(M) \cong \mrm{L}\widehat{\Lambda}_{\a}(I) \cong \mrm{R}\opn{Hom}_A(\opn{Tel}(A;\mathbf{a})\otimes_A \widehat{A},I) \cong
\opn{Hom}_A(\opn{Tel}(A;\mathbf{a}),\opn{Hom}_A(\widehat{A},I)),
\]
while by \cite[Corollary 5.25]{PSY1}, the base change property of the telescope complex, 
and the fact that $\widehat{\a}$ is weakly proregular, 
there are natural isomorphisms
\[
\mrm{L}\Lambda_{\widehat{\a}} ( \mrm{R}\widehat{\Gamma}_{\a} (M) ) \cong
\mrm{R}\opn{Hom}_{\widehat{A}}(\opn{Tel}(A;\mathbf{a})\otimes_A \widehat{A},\widehat{\Gamma}_{\a}(I) ) \cong \opn{Hom}_A(\opn{Tel}(A;\mathbf{a}),\widehat{\Gamma}_{\a}(I)).
\]
It follows that it is enough to show that there is a natural $\widehat{A}$-linear quasi-isomorphism
\[
\opn{Hom}_A(\opn{Tel}(A;\mathbf{a}),\widehat{\Gamma}_{\a}(I)) \to \opn{Hom}_A(\opn{Tel}(A;\mathbf{a}),\opn{Hom}_A(\widehat{A},I)). 
\]
Applying the functor $\opn{Hom}_A(\opn{Tel}(A;\mathbf{a}),-)$ to the map $\chi$ from (\ref{eqn:tor-chai}), 
we obtain a natural $\widehat{A}$-linear map
\begin{equation}\label{eqn:before-rest-tor}
\opn{Hom}_A(\opn{Tel}(A;\mathbf{a}),\widehat{\Gamma}_{\a}(I)) \to \opn{Hom}_A(\opn{Tel}(A;\mathbf{a}),\opn{Hom}_A(\widehat{A},I)). 
\end{equation}
To show that this map is a quasi-isomorphism, we apply the forgetful functor 
\[
\opn{Rest}_{\widehat{A}/A} : \cat{D}(\widehat{A}) \to \cat{D}(A),
\]
and obtain an $A$-linear map
\begin{equation}\label{eqn:after-rest-tor}
\opn{Hom}_A(\opn{Tel}(A;\mathbf{a}),\Gamma_{\a}(I)) \to \opn{Hom}_A(\opn{Tel}(A;\mathbf{a}),\opn{Hom}_A(\widehat{A},I))
\end{equation}
which by (\ref{eqn:composition-is-tor}) fits into the commutative diagram
\begin{equation}\label{eqn:f-diag-tor}
\begin{tikzcd}[column sep=huge, row sep=huge]
\opn{Hom}_A(\opn{Tel}(A;\mathbf{a}),I) & \opn{Hom}_A(\opn{Tel}(A;\mathbf{a}),\opn{Hom}_A(\widehat{A},I)) \ar{l}\\
\opn{Hom}_A(\opn{Tel}(A;\mathbf{a}),\Gamma_{\a}(I))\ar{u}\ar{ur}
\end{tikzcd}
\end{equation}

By the proof of Lemma \ref{lem:tensor-completion-with-koszul},
the map $\opn{Tel}(A;\mathbf{a}) \to \opn{Tel}(A;\mathbf{a})\otimes_A \widehat{A}$ 
induced by the completion map $A \to \widehat{A}$ is a quasi-isomorphism,
so by adjunction the horizontal map in (\ref{eqn:f-diag-tor}) is also a quasi-isomorphism.
By \cite[Proposition 5.8]{PSY1}, the complex $\opn{Tel}(A;\mathbf{a})$ is a K-projective resolution of $\mrm{R}\Gamma_{\a}(A)$, so the first isomorphism of \cite[Theorem 7.12]{PSY1} implies that the vertical map in (\ref{eqn:f-diag-tor}) is a quasi-isomorphism.
It follows that (\ref{eqn:before-rest-tor}) is a quasi-isomorphism.
\end{proof}

\begin{rem}\label{rem:MGM}
The Matlis-Greenlees-May equivalence is equivalent to the fact that given a commutative ring $A$ and a finitely generated weakly proregular ideal $\a\subseteq A$, 
the diagram
\[
\begin{tikzcd}[column sep=huge, row sep=huge]
\cat{D}(A)_{\a-\opn{tor}}  
\arrow{r}{\mrm{L}\Lambda_{\a}}
\arrow[loop left]{}{1_{\cat{D}(A)_{\a-\opn{tor}}}}     
& \cat{D}(A)_{\a-\opn{com}}  \arrow[bend right=-50]{l}{\mrm{R}\Gamma_{\a}}
\arrow[loop right]{}{1_{\cat{D}(A)_{\a-\opn{com}}}}     
\end{tikzcd}
\]
is commutative up to natural isomorphisms, where
$1_{\cat{D}(A)_{\a-\opn{tor}}}$ and $1_{\cat{D}(A)_{\a-\opn{com}}}$ are the identity functors on each of these categories.
Explicitly, this means that there are natural isomorphisms
\[
\mrm{L}\Lambda_{\a}(\mrm{R}\Gamma_{\a}(M)) \cong M, \quad \mrm{R}\Gamma_{\a}(\mrm{L}\Lambda_{\a}(N)) \cong N,
\]
for any $M \in \cat{D}(A)_{\a-\opn{com}}$ and any $N \in \cat{D}(A)_{\a-\opn{tor}}$.
Theorems \ref{thm:RGammaLLambda} and \ref{thm:LLambdaRGamma} enlarge this picture, 
and imply that given a commutative ring $A$ and a finitely generated weakly proregular ideal $\a\subseteq A$ such that $\widehat{\a}:=\widehat{A}\cdot \a$ is also weakly proregular, the diagram
\[
\begin{tikzcd}[column sep=huge, row sep=huge]
& \cat{D}(A) 
\arrow{dl}[swap]{\mrm{R}\widehat{\Gamma}_{\a}}
\arrow{dr}{\mrm{L}\widehat{\Lambda}_{\a}} & \\
\cat{D}(\widehat{A})_{\widehat{\a}-\opn{tor}}  
\arrow{rr}{\mrm{L}\Lambda_{\widehat{\a}}} 
\arrow[loop left]{}{1_{\cat{D}(\widehat{A})_{\widehat{\a}-\opn{tor}}}}
& & \cat{D}(\widehat{A})_{\widehat{\a}-\opn{com}} 
\arrow[bend right=-50]{ll}{\mrm{R}\Gamma_{\widehat{\a}}}
\arrow[loop right]{}{1_{\cat{D}(\widehat{A})_{\widehat{\a}-\opn{com}}}}
\end{tikzcd}
\]
is commutative up to natural isomorphisms.

\end{rem}

\begin{rem}
In case $A$ is noetherian,
Theorems \ref{thm:RGammaLLambda} and \ref{thm:LLambdaRGamma} were proved independently by 
Sather-Wagstaff and Wicklein (\cite[Section 4]{SWW2}). 
In this case $\widehat{A}$ is flat over $A$, 
and this fact is crucially used in their proofs.
For our applications below,
however, we cannot assume that $A$ is noetherian (as enveloping algebras of adic noetherian algebras seldom are), so the more complicated proofs given above are essential for our needs.
\end{rem}

\section{Adic reduction to the diagonal}\label{sec:ARD}

We are now able to answer the first question of Section \ref{sec:quest}. 
Recall that a pair $(A,\a)$ is called an adic noetherian ring if $A$ is a noetherian ring,
$\a\subseteq A$ is an ideal,
and the natural map $A \to \Lambda_{\a}(A)$ is an isomorphism.
Given a noetherian ring $\K$, 
an adic noetherian $\K$-algebra $(A,\a)$ is called \textbf{formally essentially of finite type} over $\K$ if $A/\a$
is essentially of finite type over $\K$ (that is, if $A/\a$ is a localization of finite type $\K$-algebra).
In particular, if $A$ is essentially of finite type over $\K$,
and $\a \subseteq A$ is an ideal, then $(\Lambda_{\a}(A),\a\cdot \Lambda_{\a}(A))$ is formally essentially of finite type over $\K$.
%,but there are other interesting examples.

\begin{thm}\label{thm:ard}
Let $\K$ be a noetherian ring, and let $(A,\a)$ be an adic noetherian $\K$-algebra which is flat and formally essentially of finite type over $\K$. 
Set 
\[
I:= \a\otimes_{\K} A + A\otimes_{\K} \a \subseteq A\otimes_{\K} A.
\]
\begin{enumerate}
\item
For any $M, N \in \cat{D}(A)$, there is a natural isomorphism
\[
A\otimes^{\mrm{L}}_{ A\widehat{\otimes}_{\K} A } \mrm{R}\widehat{\Gamma}_I ( M\otimes^{\mrm{L}}_{\K} N ) \cong \mrm{R}\Gamma_{\a}(M \otimes^{\mrm{L}}_A N)
\]
in $\cat{D}(A)$. If moreover either $M$ or $N$ is cohomologically $\a$-torsion, then
there is a natural isomorphism
\[
A\otimes^{\mrm{L}}_{ A\widehat{\otimes}_{\K} A } \mrm{R}\widehat{\Gamma}_I ( M\otimes^{\mrm{L}}_{\K} N ) \cong M \otimes^{\mrm{L}}_A N
\]
in $\cat{D}(A)$.
\item
For any $M, N \in \cat{D}(A)$, there is a natural isomorphism
\[
\mrm{L}\Lambda_{\a} (A\otimes^{\mrm{L}}_{ A\widehat{\otimes}_{\K} A } \mrm{L}\widehat{\Lambda}_I ( M\otimes^{\mrm{L}}_{\K} N )) \cong \mrm{L}\Lambda_{\a}(M \otimes^{\mrm{L}}_A N)
\]
in $\cat{D}(A)$.
\item
For any $M,N \in \cat{D}^{-}_{\mrm{f}}(A)$, there is a natural isomorphism
\[
A\otimes^{\mrm{L}}_{ A\widehat{\otimes}_{\K} A } \mrm{L}\widehat{\Lambda}_I ( M\otimes^{\mrm{L}}_{\K} N ) \cong M \otimes^{\mrm{L}}_A N
\]
in $\cat{D}(A)$.
\end{enumerate}
\end{thm}
\begin{proof}
\leavevmode
\begin{enumerate}
\item According to \cite[Theorem 2.6]{Sh2}, 
since $\K$ is noetherian and $(A,\a)$ is flat and formally essentially of finite type over $\K$,
the ideal $I$ is weakly proregular.
Since $A$ is $\a$-adically complete, we have that $\mrm{R}\widehat{\Gamma}_{\a} = \mrm{R}\Gamma_{\a}$, 
so by \cite[Theorem 5.1]{Sh2}, there is a natural isomorphism
\[
A\otimes^{\mrm{L}}_{A\widehat{\otimes}_{\K} A}  \mrm{R}\widehat{\Gamma}_{I} (M\otimes^{\mrm{L}}_{\K} N) 
\cong
\mrm{R}\Gamma_{\a} \left( 
A\otimes^{\mrm{L}}_{A\otimes_{\K} A}  (M\otimes^{\mrm{L}}_{\K} N) 
\right).
\]
By classical reduction to the diagonal we get a natural isomorphism
\[
A\otimes^{\mrm{L}}_{A\widehat{\otimes}_{\K} A}  \mrm{R}\widehat{\Gamma}_{I} (M\otimes^{\mrm{L}}_{\K} N) 
\cong
\mrm{R}\Gamma_{\a} \left( M\otimes^{\mrm{L}}_A N \right).
\]
If either $M$ or $N$ is cohomologically $\a$-torsion, 
it follows from \cite[Corollary 4.26]{PSY1} that
\[
\mrm{R}\Gamma_{\a}(M)\otimes^{\mrm{L}}_A N \cong \mrm{R}\Gamma_{\a}(M\otimes^{\mrm{L}}_A N),
\]
so we obtain the second claim of (1).
\item By \cite[Lemma 7.2]{PSY1}, we have a natural isomorphism
\[
\mrm{L}\Lambda_{\a} \left( 
A\otimes^{\mrm{L}}_{A\widehat{\otimes}_{\K} A} ( \mrm{L}\widehat{\Lambda}_I(M\otimes^{\mrm{L}}_{\K} N))
\right) \cong
\mrm{L}\Lambda_{\a} \mrm{R}\Gamma_{\a} \left( 
A\otimes^{\mrm{L}}_{A\widehat{\otimes}_{\K} A} ( \mrm{L}\widehat{\Lambda}_I(M\otimes^{\mrm{L}}_{\K} N))
\right).
\]
The fact that $A/\a$ is essentially of finite type over $\k$ implies that
\[
(A\otimes_{\K} A)/I \cong A/\a\otimes_{\K} A/\a
\]
is noetherian. Hence, by \cite[Corollary 2 after Proposition III.2.11.14]{Bou},
the ring $A\widehat{\otimes}_{\K} A$ is also noetherian.
Letting $\widehat{I} := (A\widehat{\otimes}_{\K} A) \cdot I$, it follows that $\widehat{I}$ is weakly proregular.
The image of $\widehat{I}$ under the natural map 
\[
A\widehat{\otimes}_{\K} A \to A
\]
is equal to $\a$. 
Hence, by \cite[Corollary 3.14]{Sh2}, we have that % Hochschild commutes with adic completion
\[
\mrm{L}\Lambda_{\a} \mrm{R}\Gamma_{\a} \left( 
A\otimes^{\mrm{L}}_{A\widehat{\otimes}_{\K} A} ( \mrm{L}\widehat{\Lambda}_I(M\otimes^{\mrm{L}}_{\K} N))
\right) \cong
\mrm{L}\Lambda_{\a}  \left( 
A\otimes^{\mrm{L}}_{A\widehat{\otimes}_{\K} A}  \mrm{R}\Gamma_{\widehat{I}} ( \mrm{L}\widehat{\Lambda}_I(M\otimes^{\mrm{L}}_{\K} N)) 
\right).
\]
We may now apply Theorem \ref{thm:RGammaLLambda}, and obtain a natural isomorphism
\[
\mrm{L}\Lambda_{\a}  \left( 
A\otimes^{\mrm{L}}_{A\widehat{\otimes}_{\K} A}  \mrm{R}\Gamma_{\widehat{I}} ( \mrm{L}\widehat{\Lambda}_I(M\otimes^{\mrm{L}}_{\K} N)) 
\right) \cong
\mrm{L}\Lambda_{\a}  \left( 
A\otimes^{\mrm{L}}_{A\widehat{\otimes}_{\K} A}  \mrm{R}\widehat{\Gamma}_{I} (M\otimes^{\mrm{L}}_{\K} N) 
\right).
\]
Using (1) we see that
\[
\mrm{L}\Lambda_{\a}  \left( 
A\otimes^{\mrm{L}}_{A\widehat{\otimes}_{\K} A}  \mrm{R}\widehat{\Gamma}_{I} (M\otimes^{\mrm{L}}_{\K} N) 
\right) \cong
\mrm{L}\Lambda_{\a} \left(
\mrm{R}\Gamma_{\a} \left( M\otimes^{\mrm{L}}_A N \right)
\right),
\]
so the result follows from applying \cite[Lemma 7.2]{PSY1} again.
\item Let $P \iso M$ and $Q \iso N$ be bounded above resolutions made of finitely generated free $A$-modules.
It follows that $P\otimes_{\K} Q$ is a bounded above complex of finitely generated free
$A\otimes_{\K} A$-modules, so that
\[
\mrm{L}\widehat{\Lambda}_I(M\otimes^{\mrm{L}}_{\K} N) \cong \widehat{\Lambda}_I(P\otimes_{\K} Q)
\]
is a bounded above complex of finitely generated free
$A\widehat{\otimes}_{\K} A$-modules.
Hence, 
\[
A\otimes^{\mrm{L}}_{A\widehat{\otimes}_{\K} A} ( \mrm{L}\widehat{\Lambda}_I(M\otimes^{\mrm{L}}_{\K} N) ) \in \cat{D}^{-}_{\mrm{f}}(A).
\]
Since $A$ is noetherian and $\a$-adically complete, any finitely generated $A$-module is $\a$-adically complete, so it follows by \cite[Theorem 1.19]{PSY2} that 
\[
A\otimes^{\mrm{L}}_{A\widehat{\otimes}_{\K} A} ( \mrm{L}\widehat{\Lambda}_I(M\otimes^{\mrm{L}}_{\K} N))
\]
is cohomologically $\a$-adically complete. Thus,
\[
A\otimes^{\mrm{L}}_{A\widehat{\otimes}_{\K} A} ( \mrm{L}\widehat{\Lambda}_I(M\otimes^{\mrm{L}}_{\K} N)) \cong 
\mrm{L}\Lambda_{\a} \left( 
A\otimes^{\mrm{L}}_{A\widehat{\otimes}_{\K} A } ( \mrm{L}\widehat{\Lambda}_I(M\otimes^{\mrm{L}}_{\K} N))
\right).
\]
Similarly, since $M, N \in \cat{D}^{-}_{\mrm{f}}(A)$, 
we have that $M\otimes^{\mrm{L}}_A N \in \cat{D}^{-}_{\mrm{f}}(A)$,
so $M\otimes^{\mrm{L}}_A N$ is also cohomologically $\a$-adically complete. 
Hence, (3) follows from (2).
\end{enumerate}
\end{proof}

\section{Cofiniteness and derived completion}\label{sec:COF}

Our final result is a positive answer to the second question of Section \ref{sec:quest}.
In the particular case where $A$ is noetherian it was also obtained independently by Sather-Wagstaff and Wicklein (\cite[Theorem 1.3]{SWW1}). Enveloping algebras like those occurring in Theorem \ref{thm:ard} are non-noetherian examples that satisfy the conditions of the next result.

\begin{thm}\label{thm:cofinite}
Let $A$ be a commutative ring, let $\a\subseteq A$ be a finitely generated weakly proregular ideal, denote by $\widehat{A}$ the $\a$-adic completion of $A$, and assume that $\widehat{A}$ is noetherian. (If $A$ is noetherian all these conditions are always satisfied for any ideal $\a\subseteq A$).
Given $M \in \cat{D}^{\mrm{b}}(A)$, the following are equivalent:
\begin{enumerate}
\item $\mrm{R}\opn{Hom}_A(A/\a,M) \in \cat{D}_{\mrm{f}}(A/\a)$.
\item $\mrm{L}\widehat{\Lambda}_{\a}(M) \in \cat{D}_{\mrm{f}}^{\mrm{b}}(\widehat{A})$.
\end{enumerate}
\end{thm}
\begin{proof}
Take some $M \in \cat{D}^{\mrm{b}}(A)$. 
According to \cite[Lemma 2.5]{Sh1}, there are natural isomorphisms
\[
\mrm{R}\opn{Hom}_{\widehat{A}}(A/\a,\mrm{R}\widehat{\Gamma}_{\a}(M)) \cong \mrm{R}\opn{Hom}_{\widehat{A}}(A/\a,\mrm{L}\widehat{\Lambda}_{\a}(M)) \cong \mrm{R}\opn{Hom}_{A}(A/\a,M)
\]
in $\cat{D}(A/\a)$. 
Now, if $\mrm{L}\widehat{\Lambda}_{\a}(M) \in \cat{D}^{\mrm{b}}_{\mrm{f}}(\widehat{A})$, then 
\[
\mrm{R}\opn{Hom}_{A}(A/\a,M) \cong \mrm{R}\opn{Hom}_{\widehat{A}}(A/\a,\mrm{L}\widehat{\Lambda}_{\a}(M)) \in \cat{D}_{\mrm{f}}(A/\a)
\]
proving one direction of the theorem.

Conversely, suppose that 
\[
\mrm{R}\opn{Hom}_A(A/\a,M) \in \cat{D}_{\mrm{f}}(A/\a).
\]
Let $\widehat{\a}:= \widehat{A}\cdot \a$.
By the above isomorphism, we have that
\[
\mrm{R}\opn{Hom}_{\widehat{A}}(A/\a,\mrm{R}\widehat{\Gamma}_{\a}(M)) \in \cat{D}_{\mrm{f}}(A/\a).
\]
Since $\a$ is weakly proregular, by \cite[Corollary 4.28]{PSY1}, the functor $\mrm{R}\Gamma_{\a}$ has finite cohomological dimension. Hence, $\mrm{R}\widehat{\Gamma}_{\a}$ also has finite cohomological dimension, so that $\mrm{R}\widehat{\Gamma}_{\a}(M)$ is bounded.
The complex $\mrm{R}\widehat{\Gamma}_{\a}(M)$ has $\a$-torsion cohomology, 
so by \cite[Corollary 4.32]{PSY1} we conclude that
\[
\mrm{R}\widehat{\Gamma}_{\a}(M) \in \cat{D}^{\mrm{b}}(\widehat{A})_{\a\opn{-tor}}.
\]
Hence, all the conditions of \cite[Theorem 3.10]{PSY2} are satisfied for the noetherian ring $\widehat{A}$ and the bounded complex $\mrm{R}\widehat{\Gamma}_{\a}(M)$, 
so we deduce that there exists some $N \in \cat{D}^{\mrm{b}}_{\mrm{f}}(\widehat{A})$ such that 
\begin{equation}\label{eqn:using-cofinite}
\mrm{R}\Gamma_{\widehat{\a}}(N) \cong \mrm{R}\widehat{\Gamma}_{\a}(M).
\end{equation}
By Theorem \ref{thm:LLambdaRGamma}, we have that
\[
\mrm{L}\widehat{\Lambda}_{\a}(M) \cong \mrm{L}\Lambda_{\widehat{\a}} \mrm{R}\widehat{\Gamma}_{\a} (M),
\]
so using (\ref{eqn:using-cofinite}), we see that
\[
\mrm{L}\widehat{\Lambda}_{\a}(M) \cong \mrm{L}\Lambda_{\widehat{\a}} \mrm{R}\Gamma_{\widehat{\a}}(N).
\]
But by \cite[Lemma 7.2]{PSY1} we have that
\[
\mrm{L}\Lambda_{\widehat{\a}} \mrm{R}\Gamma_{\widehat{\a}}(N) \cong \mrm{L}\Lambda_{\widehat{\a}}(N),
\]
and since $N \in \cat{D}^{\mrm{b}}_{\mrm{f}}(\widehat{A})$,
and $\widehat{A}$ is noetherian and $\widehat{\a}$-adically complete,
by \cite[Proposition 3.1]{PSY2}, 
\[
\mrm{L}\Lambda_{\widehat{\a}}(N) \cong N,
\]
so that
\[
\mrm{L}\widehat{\Lambda}_{\a}(M) \cong N \in \cat{D}^{\mrm{b}}_{\mrm{f}}(\widehat{A}).
\]
\end{proof}

\textbf{Acknowledgments.}
The author would like to thank Srikanth Iyengar, Sean Sather-Wagstaff and Amnon Yekutieli for some useful discussions.
The author is grateful to the anonymous referee for many helpful comments and suggestions that helped improving this manuscript.

\end{document}